\documentclass[preprint,12pt]{elsarticle}
\usepackage{amsfonts}
\usepackage{amsmath}
\usepackage{amssymb}
\usepackage{graphicx}
\usepackage[tableposition=top]{caption}
\usepackage{subcaption}
\usepackage{float}
\usepackage{setspace}
\usepackage{placeins}
\usepackage{mathrsfs}

\newsavebox \foobox
\newlength{\foodim}

\graphicspath{{C:/Figures/}}
\newtheorem{theorem}{Theorem}
\newtheorem{proof}{Proof}

\newtheorem{corollary}{Corollary}

\newtheorem{definition}{Definition}

\newtheorem{lemma}{Lemma}

\newtheorem{proposition}{Proposition}

\numberwithin{equation}{section}
\journal{Computers \& Mathematics with Applications}

\input{tcilatex}

\begin{document}

\begin{frontmatter}
\title{On the asymptotic stability of the time--fractional Lengyel--Epstein system}
\author[a]{Djamel Mansouri}
\author[b]{Salem Abdelmalek}
\author[c]{Samir Bendoukha}
\address[a]{Department of Mathematics, ICOSI laboratory, University Abbes Laghrour, Khenchela, Algeria}
\address[b]{Department of Mathematics and Informatics, Faculty of Exact Sciences and Natural Sciences and Life, University of Larbi Tebessi,Tebessa,12002, Algeria}
\address[c]{Electrical Engineering Department, College of Engineering at Yanbu, Taibah University, Saudi Arabia. E-mail address: sbendoukha@taibahu.edu.sa}

\begin{abstract}
This paper concerns a time fractional version of the conventional Lengyel--Epstein CIMA reaction model. We define the invariant regions of the system and establish sufficient conditions for the unique equilibrium's local and global asymptotic stability. Numerical results are presented to illustrate the effect of the fractional order on system dynamics.
\end{abstract}
\begin{keyword}
Fractional calculus, fractional Lengyel--Epstein system, asymptotic stability, fractional Lyapunov method.
\end{keyword}
\end{frontmatter}

\section{Introduction}

In this paper, we are interested in a fractional version of the
Lengyel--Epstein reaction--diffusion system proposed in \cite%
{Lengyel1992,Lengyel1991} as a model of the chlorite--iodide malonic--acid
(CIMA) chemical reaction \cite{DeKepper1990}. The considered model has
attracted the interest of many researchers since its inception in 1991. The
reason for this interest is the fact that the CIMA\ reaction is one of the
earliest experiments that confirmed the theoretical propositions of Alan
Turing in 1952 \cite{Turing1952} concerning the chemical basis for
morphogenesis and more generally pattern formation. The CIMA\ reaction can
be described by three chemical reaction schemes as follows%
\begin{equation}
\left \{ 
\begin{array}{l}
MA+I_{2}\rightarrow IMA+I^{-}+H^{+}, \\ 
CIO_{2}+I^{-}\rightarrow \frac{1}{2}I_{2}+CIO_{2}^{-}, \\ 
CIO_{2}^{-}+4I^{-}+4H^{+}\rightarrow CI^{-}+2I_{2}+2H_{2}O.%
\end{array}%
\right.  \label{0.1}
\end{equation}%
Considering the empirical rate laws corresponding to these processes and
ignoring constant factors, the model for this reaction was reduced to the
conventional Lengyel--Epstein model with two dependent variables $u$ and $v$
representing the time evolution of the concentrations of $\left[ I^{-}\right]
$ and $\left[ CIO_{2}^{-}\right] $, respectively. The general dynamics of
the Lengyel--Epstein system have been examined in a number of studies.
Sufficient conditions for its local and global asymptotic stability can be
found in \cite{Ni2005,Yi2008,Yi2009,Lisena2014}. In \cite{Yi2008,Wang2013},
the authors establish sufficient conditions for the Turing or
diffusion--driven instability of the system. More details on the formation
of patterns in the Lengyel--Epstein model can be found in \cite{Silva2015}.
Also, results related to the Hopf--bifurcation for the Lengyel--Epstein
system are presented and analyzed in \cite{Jang2004,Yi2008,Wang2013}. In
addition, many studies have also examined modified versions of the system
including \cite%
{Horvath2000,Rudiger2003,Miguez2006,Vazquez2012,Scholz2009,Zheng2014,Gambino2014,Zheng2016,Wei2017,Abdelmalek2017a,Abdelmalek2017b,Abdelmalek2018a,Abdelmalek2018b}
with the aim of relaxing existing asymptotic stability and Turing
instability conditions.

In \cite{Liu2017}, the authors considered the model%
\begin{equation}
\left \{ 
\begin{array}{l}
\dfrac{\partial u}{\partial t}=\nabla ^{\gamma }u+a-u-\dfrac{4uv}{1+u^{2}},
\\ 
\dfrac{\partial v}{\partial t}=\sigma \left[ c\nabla ^{\gamma }v+b\left( u-%
\dfrac{uv}{1+u^{2}}\right) \right] ,%
\end{array}%
\right.  \label{0.3}
\end{equation}%
which accounts for anomalous diffusion in a fractal medium for example. The
term $\nabla ^{\gamma }$ denotes the Riesz fractional operator with $%
1<\gamma <2$. The authors established sufficient conditions for the
existence of Turing patterns and examined their nature. Note that system (%
\ref{0.3}) is fractional in the spatial sense. In our work, we aim to
propose and study the dynamics of the time--fractional system corresponding
to the Lengyel--Epstein model.

The following section states some of the necessary notation and theory
related to fractional systems. Section \ref{SecModel} describes the proposed
system and examines its invariant regions. Section \ref{SecStab} establishes
conditions for the asymptotic stability of the proposed system. Section \ref%
{SecNum} illustrates the analytical conditions through numerical examples.
Finally, Section \ref{SecSum} summarizes the findings of this study and
poses open questions for future investigation.

\section{Fractional Calculus\label{SecPrelim}}

In this section, we start with some of the necessary notation and stability
theory related to the subject.

\begin{definition}
\label{Def1}\cite{Podlubny1999} The Riemann--Liouville fractional derivative
of order $\delta $\ of an integrable function $f\left( t\right) $ is defined
as%
\begin{equation}
_{t_{0}}D_{t}^{-\delta }f\left( t\right) =\frac{1}{\Gamma \left( \delta
\right) }\int_{t_{0}}^{t}\frac{f\left( \tau \right) }{\left( t-\tau \right)
^{1-\delta }}d\tau .  \label{1.1}
\end{equation}%
where $0<\delta \in 
\mathbb{R}
^{+}$ and $\Gamma \left( \delta \right) =\int_{0}^{\infty }e^{-t}t^{\delta
-1}dt$ is the Gamma function.
\end{definition}

\begin{definition}
\label{Def2}\cite{Kilbas2006} The Caputo fractional derivative of order $%
\delta >0$ of a function $f\ $of class $C^{n}$ for $t>t_{0}$ is defined as%
\begin{equation}
_{t_{0}}^{C}D_{t}^{\delta }f\left( t\right) =\frac{1}{\Gamma \left( n-\delta
\right) }\int_{t_{0}}^{t}\frac{f^{\left( n\right) }\left( \tau \right) }{%
\left( t-\tau \right) ^{\delta -n-1}}d\tau ,  \label{1.2}
\end{equation}%
with $n=\min \left \{ k\in 
\mathbb{N}
\ |\ k>\delta \right \} $ and $\Gamma $\ representing the gamma function.
\end{definition}

Note that the constant $\left( u^{\ast },v^{\ast }\right) $ is an
equilibrium for the Caputo fractional non--autonomous dynamic system%
\begin{equation}
\left \{ 
\begin{array}{l}
_{t_{0}}^{C}D_{t}^{\delta }u=F\left( u,v\right) ,\text{ \ \ \ \ \ \ in }%
\mathbb{R}^{+},\medskip \\ 
_{t_{0}}^{C}D_{t}^{\delta }v=G\left( u,v\right) ,\text{ \ \ \ \ \ in }%
\mathbb{R}^{+},%
\end{array}%
\right.  \label{1.6}
\end{equation}%
if and only if 
\begin{equation}
F\left( u^{\ast },v^{\ast }\right) =G\left( u^{\ast },v^{\ast }\right) =0.
\label{1.7}
\end{equation}%
The following lemmas hold.

\begin{lemma}
\label{Lemma2}Let $u\left( t\right) $ be a continuous and differentiable
real function. For any time instant $t\geq t_{0}$,%
\begin{equation}
_{t_{0}}^{C}D_{t}^{\delta }u^{2}\left( t\right) \leq 2u\left( t\right)
_{t_{0}}^{C}D_{t}^{\delta }u\left( t\right) ,  \label{1.8}
\end{equation}%
with $\delta \in \left( 0,1\right] $.
\end{lemma}

\begin{lemma}
\label{Lemma3}\cite{Matignon1996} An equilibrium point $\left( u^{\ast
},v^{\ast }\right) $ of (\ref{1.6}) is locally asymptotically stable iff%
\begin{equation}
\left \vert \arg \left( \lambda _{i}\right) \right \vert >\frac{\delta \pi }{%
2},\ \ \ i=1,2,  \label{1.9}
\end{equation}%
where $\lambda _{i}$ are the eigenvalues of the Jacobian matrix $J\left(
u^{\ast },v^{\ast }\right) $ and $\arg \left( \cdot \right) $ denotes the
argument of a complex number.
\end{lemma}

\begin{lemma}
\label{Lemma4}If an equilibrium point $\left( u^{\ast },v^{\ast }\right) $
of (\ref{1.6}) is locally asymptotically stable for the standard system%
\begin{equation}
\left \{ 
\begin{array}{l}
u_{t}=F\left( u,v\right) ,\text{ \ \ \ \ \ \ in }\mathbb{R}^{+},\medskip \\ 
v_{t}=G\left( u,v\right) ,\text{ \ \ \ \ \ in }\mathbb{R}^{+},%
\end{array}%
\right.  \label{1.10}
\end{equation}%
then, it is also locally asymptotically stable for\ (\ref{1.6}).
\end{lemma}

\begin{proof}
Assuming that $\left( u^{\ast },v^{\ast }\right) $ is a locally
asymptotically stable equilibrium for (\ref{1.10}), then all the eigenvalues
of the Jacobian matrix have negative real parts, i.e.%
\begin{equation*}
\left \vert \arg \left( \lambda _{i}\right) \right \vert >\frac{\pi }{2},\ \
\ i=1,2.
\end{equation*}%
Since $\delta <1$, it is trivial to see that (\ref{1.9}) holds, which leads
to the local asymptotic stability of $\left( u^{\ast },v^{\ast }\right) $ as
an equilibrium of (\ref{1.6}).
\end{proof}

\begin{corollary}
\label{Corr1}In the diffusion case, if an equilibrium point $\left( u^{\ast
},v^{\ast }\right) $ of (\ref{1.6}) is locally asymptotically stable for the
integer system%
\begin{equation*}
\left \{ 
\begin{array}{l}
u_{t}-d_{1}\Delta u=F\left( u,v\right) ,\text{ \ \ \ \ \ \ in }\mathbb{R}%
^{+}\times \Omega ,\medskip \\ 
v_{t}-d_{2}\Delta v=G\left( u,v\right) ,\text{ \ \ \ \ \ in }\mathbb{R}%
^{+}\times \Omega ,%
\end{array}%
\right.
\end{equation*}%
then it is also locally asymptotically stable for%
\begin{equation*}
\left \{ 
\begin{array}{l}
_{0}^{C}D_{t}^{\delta }u-d_{1}\Delta u=F\left( u,v\right) ,\text{ \ \ \ \ \
\ in }\mathbb{R}^{+}\times \Omega ,\medskip \\ 
_{0}^{C}D_{t}^{\delta }v-d_{2}\Delta v=G\left( u,v\right) ,\text{ \ \ \ in }%
\mathbb{R}^{+}\times \Omega .%
\end{array}%
\right.
\end{equation*}
\end{corollary}

\section{System Model\label{SecModel}}

In this paper, we consider the time fractional Lengyel--Epstein system%
\begin{equation}
\left\{ 
\begin{array}{l}
_{0}^{C}D_{t}^{\delta }u-d_{1}\Delta u=a-u-\frac{4uv}{1+u^{2}}=:F\left(
u,v\right) ,\text{ \ \ \ \ \ \ in }\mathbb{R}^{+}\times \Omega ,\medskip  \\ 
_{0}^{C}D_{t}^{\delta }v-d_{2}\Delta v=\sigma b\left( u-\frac{uv}{1+u^{2}}%
\right) =:G\left( u,v\right) ,\text{ \ \ \ in }\mathbb{R}^{+}\times \Omega ,%
\end{array}%
\right.   \label{2.1}
\end{equation}%
where $\Omega $ is a bounded domain in $%
\mathbb{R}
^{n}$ ($n=2,3$ in practice) with smooth boundary $\partial \Omega $, $\Delta
=\underset{i=1}{\overset{n}{\sum }}\frac{\partial ^{2}}{\partial x_{i}^{2}}$%
, $0<\delta \leq 1$\ is the fractional order, $_{0}^{C}D_{t}^{\delta }$
denotes the Caputo fractional derivative over $\left( 0,\infty \right) $ as
defined in (\ref{1.2}), and $d_{1},d_{2},a$ and $\sigma $ are strictly
positive constants. We assume the nonnegative initial conditions%
\begin{equation}
0\leq u\left( 0,x\right) =u_{0}\left( x\right) ,\text{ \ }0\leq v\left(
0,x\right) =v_{0}\left( x\right) ,\text{\ \ \ \ in }\Omega ,  \label{2.1.0}
\end{equation}%
with $u_{0},v_{0}\in C^{2}\left( \Omega \right) \cap C\left( \overline{%
\Omega }\right) $, and impose homogeneous Neumann boundary conditions%
\begin{equation}
\dfrac{\partial u}{\partial \nu }=\dfrac{\partial v}{\partial \nu }=0\ \ 
\text{\ \ \ \ \ on \ \ \ }\mathbb{R}^{+}\times \partial \Omega ,
\label{2.1.1}
\end{equation}%
where $\nu $ is the unit outer normal to $\partial \Omega $.

Before we study the local and global asymptotic stability of the solutions
of the proposed system, let us define its invariant region. We start with a
definition of the term invariant region following the lines of \cite%
{Mottoni1979,Yi2009}. Note that when $F\left( u,v\right) =0$, the curves in
the $u$--$v$ plane are called $u$--isoclines. Similarly, they are called $v$%
--isoclines when $G\left( u,v\right) =0$. in addition, if the vector field $%
\left( F,G\right) $ does not point outwards at the boundary of a certain
rectabgle $\partial \Re $, then $\Re $ is said to be an invariant rectangle.
This is similar to following definition.

\begin{definition}
A rectangle $\Re $ is said to be an invariant rectangle if the vector field $%
\left( F,G\right) $ on the boundary $\partial \Re $ points inside, i.e.%
\begin{equation}
\left \{ 
\begin{array}{l}
F\left( 0,v\right) \geq 0\text{ and }F\left( r_{1},v\right) \leq 0\text{ for 
}0<v<r_{2},\medskip \\ 
G\left( u,0\right) \geq 0\text{ and }G\left( u,r_{2}\right) \leq 0\text{ for 
}0<u<r_{1}.%
\end{array}%
\right.  \label{2.2}
\end{equation}
\end{definition}

The following proposition describes the invariant region of the proposed
system (\ref{2.1}).

\begin{proposition}
\label{Prop0}System (\ref{2.1}) admits the region of attraction%
\begin{equation}
\Re _{a}=\left( 0,a\right) \times \left( 0,1+a^{2}\right) .  \label{2.3}
\end{equation}
\end{proposition}

\section{Asymptotic Stability Conditions\label{SecStab}}

\subsection{Local Stability}

In this section, we derive sufficient conditions for the local asymptotic
stability of the equilibrium point of (\ref{2.1}). The free diffusions
system corresponding to (\ref{2.1}) is%
\begin{equation}
\left \{ 
\begin{array}{l}
_{0}^{C}D_{t}^{\delta }u=a-u-\frac{4uv}{1+u^{2}},\medskip \\ 
_{0}^{C}D_{t}^{\delta }v=\sigma b\left( u-\frac{uv}{1+u^{2}}\right) .%
\end{array}%
\right.  \label{3.0}
\end{equation}

\begin{proposition}
\label{Prop1}System (\ref{3.0}) has the unique equilibrium%
\begin{equation}
\left( u^{\ast },v^{\ast }\right) =\left( \alpha ,1+\alpha ^{2}\right) ,
\label{3.1}
\end{equation}%
with 
\begin{equation}
\alpha =\frac{a}{5}.  \label{3.2}
\end{equation}%
Subject to 
\begin{equation*}
\Upsilon =\left( \frac{3\alpha ^{2}-5-\sigma b\alpha }{1+\alpha ^{2}}\right)
^{2}-20\frac{\sigma b\alpha }{\alpha ^{2}+1}\geq 0,
\end{equation*}%
$\left( u^{\ast },v^{\ast }\right) $ is asymptotically stable if%
\begin{equation*}
\text{tr}J<0,
\end{equation*}%
and unstable if%
\begin{equation*}
\text{tr}J>0,
\end{equation*}%
where%
\begin{equation*}
J=\left( 
\begin{array}{cc}
\frac{3\alpha ^{2}-5}{1+\alpha ^{2}} & -\frac{4\alpha }{1+\alpha ^{2}} \\ 
\sigma b\frac{2\alpha ^{2}}{1+\alpha ^{2}} & -\sigma b\frac{\alpha }{%
1+\alpha ^{2}}%
\end{array}%
\right) .
\end{equation*}%
Alternatively, if $\Upsilon <0$, then $\left( u^{\ast },v^{\ast }\right) $\
is asymptotically stable whenever tr$J\leq 0$ or%
\begin{equation}
\left \vert \arg \left( \lambda _{1}\right) \right \vert >\delta \frac{\pi }{%
2}\text{ and }\left \vert \arg \left( \lambda _{2}\right) \right \vert
>\delta \frac{\pi }{2},  \label{3.3}
\end{equation}%
where%
\begin{equation}
\lambda _{1,2}=\frac{1}{2}\left[ \frac{3\alpha ^{2}-5-\sigma b\alpha }{%
1+\alpha ^{2}}\pm i\sqrt{-\Upsilon }\right] .  \label{3.4}
\end{equation}
\end{proposition}

\begin{proof}
The Jacobian matrix in $\left( u^{\ast },v^{\ast }\right) $ is given by%
\begin{equation*}
J\left( u^{\ast },v^{\ast }\right) =\left( 
\begin{array}{cc}
\frac{3\alpha ^{2}-5}{1+\alpha ^{2}} & -\frac{4\alpha }{1+\alpha ^{2}} \\ 
\sigma b\frac{2\alpha ^{2}}{1+\alpha ^{2}} & -\sigma b\frac{\alpha }{%
1+\alpha ^{2}}%
\end{array}%
\right) .
\end{equation*}%
Its determinant and trace are given by%
\begin{equation*}
\det J\left( u^{\ast },v^{\ast }\right) =5\sigma b\frac{\alpha }{\alpha
^{2}+1},
\end{equation*}%
and%
\begin{equation*}
\text{tr}J\left( u^{\ast },v^{\ast }\right) =\frac{3\alpha ^{2}-5-\sigma
b\alpha }{1+\alpha ^{2}},
\end{equation*}%
respectively.

The characteristic equation of the Jacobian matrix is%
\begin{equation*}
\lambda ^{2}-\left( \text{tr}J\right) \lambda +\det J=0,
\end{equation*}%
and its discriminant is%
\begin{equation*}
\Upsilon =\left( \text{tr}J\right) ^{2}-4\det J.
\end{equation*}%
We study the different cases separately. First, if $\Upsilon >0$, then the
eigenvalues $\lambda _{1,2}$ are real and can be rewritten as%
\begin{equation*}
\lambda _{1,2}=\frac{1}{2}\left[ \text{tr}J\pm \sqrt{\Upsilon }\right] .
\end{equation*}%
Note that $\det J>0$. Hence, the negativity of the eigenvalues rests on the
sign of the trace $\text{tr}J$:

\begin{itemize}
\item If $\text{tr}J<0$, then 
\begin{equation*}
\lambda_{1}=\frac{1}{2}\left[ \text{tr}J-\sqrt{\Upsilon}\right] <0,
\end{equation*}
and, therefore, $\arg \left( \lambda_{1}\right) =\pi$. Since both
eigenvalues are real, the trace is negative, and the determinant is
positive, it is evident that $\left \vert \arg \left( \lambda_{2}\right)
\right \vert =\left \vert \arg \left( \lambda_{1}\right) \right \vert =\pi>%
\frac{\delta \pi }{2}$ as $\delta \in \left( 0,1\right] $. It follows that
the equilibrium $\left( u^{\ast},v^{\ast}\right) $ is asymptotically stable.

\item If $\text{tr}J>0$, we have%
\begin{equation*}
\text{tr}J-\sqrt{\Upsilon }>0,
\end{equation*}%
leading to%
\begin{equation*}
\lambda _{1}=\frac{1}{2}\left[ \text{tr}J-\sqrt{\Upsilon }\right] >0,
\end{equation*}%
and thus%
\begin{equation*}
\left \vert \arg \left( \lambda _{1}\right) \right \vert =0.
\end{equation*}%
So, $\left( u^{\ast },v^{\ast }\right) $ is asymptotically unstable.

\item If $\text{tr}J=0$, then 
\begin{equation*}
\Upsilon >0\Rightarrow -4\det J>0,
\end{equation*}%
which is a contradiction. Hence, this case does not show up.
\end{itemize}

Next, we consider the case of the discriminant $\Upsilon $ being equal to
zero. Since $\det J>0$, then it is impossible that $\text{tr}J=0$. The
eigenvalues reduce to%
\begin{equation*}
\lambda _{1,2}=\frac{1}{2}\text{tr}J.
\end{equation*}%
The sign of the eigenvalues is identical to that of the trace. Consequently, 
$\left( u^{\ast },v^{\ast }\right) $ is asymptotically stable for all $%
\delta \in \left( 0,1\right] $ if $\text{tr}J<0$ and unstable if $\text{tr}%
J>0$.

Finally, if the discriminant $\Upsilon <0$, then%
\begin{align*}
\lambda _{1,2}& =\frac{1}{2}\left[ \text{tr}J\pm \sqrt{\Upsilon }\right] \\
& =\frac{1}{2}\left[ \text{tr}J\pm i\sqrt{-\Upsilon }\right] .
\end{align*}%
We, now, have three cases:

\begin{itemize}
\item If tr$J<0$, then by means of Lemma \ref{Lemma4}, $\left( u^{\ast
},v^{\ast }\right) $ is asymptotically stable.

\item If $\text{tr}J=0$, then%
\begin{equation*}
\left \vert \arg \left( \lambda _{1,2}=\pm \frac{1}{2}i\sqrt{-\Upsilon }%
\right) \right \vert =\frac{\pi }{2}.
\end{equation*}%
Hence, for\ $\delta <1$, $\left( u^{\ast },v^{\ast }\right) $\ is
asymptotically stable.

\item If tr$J>0$, then $\left( u^{\ast },v^{\ast }\right) $ is
asymptotically stable subject to (\ref{3.3}).
\end{itemize}

The proof is complete.
\end{proof}

Now, let us move on to the complete system (\ref{2.1}). For this, we are
going to use the eigenfunction expansion method \cite{Casten1977}. We denote
the eigenvalues of the spectral problem with Neumann boundary conditions by $%
0=\lambda _{0}\leq \lambda _{1}\leq \lambda _{2}\leq \cdots \leq \lambda
_{k}\leq \cdots $ and the corresponding normalized eigenfunctions by $\phi
_{0},\cdots ,\phi _{k},\cdots $. Let us set%
\begin{equation}
J_{i}=\left( 
\begin{array}{cc}
F_{0}-d_{1}\lambda _{i} & F_{1} \\ 
\sigma G_{0} & \sigma G_{1}-d_{2}\lambda _{i}%
\end{array}%
\right) ,  \label{3.6.1}
\end{equation}%
and%
\begin{equation}
L=\left( 
\begin{array}{cc}
d_{1}\Delta +F_{0} & F_{1} \\ 
\sigma G_{0} & d_{2}\Delta +\sigma G_{1}%
\end{array}%
\right) ,  \label{3.6.2}
\end{equation}%
where%
\begin{equation}
F_{0}=\frac{3\alpha ^{2}-5}{1+\alpha ^{2}},\ F_{1}=-\frac{4\alpha }{1+\alpha
^{2}},\ G_{0}=b\frac{2\alpha ^{2}}{1+\alpha ^{2}},\ \text{and}\ G_{1}=-b%
\frac{\alpha }{1+\alpha ^{2}}.  \label{3.6.3}
\end{equation}%
In addittion, if\ $d_{1}>d_{2}$, we define $\lambda _{01}<\lambda _{02}$\ as
the roots of%
\begin{equation}
\Upsilon _{i}=\left( d_{1}-d_{2}\right) ^{2}\lambda _{i}^{2}+2\left(
d_{1}-d_{2}\right) \left( -F_{0}+\sigma G_{1}\right) \lambda _{i}+\Upsilon .
\label{3.6.4}
\end{equation}%
The following proposition describes the conditions for the asymptotic
stability of the steady state assuming $F_{0}>0$.

\begin{proposition}
\label{Prop2}If $d_{1}=d_{2}$, then the asymptotic stability conditions are
identical to the free diffusions case as stated in Proposition \ref{Prop1}.
Alternatively, if $d_{1}\not=d_{2}$, $\text{tr}J<0$ and $\Upsilon >0$, then $%
\left( u^{\ast },v^{\ast }\right) $ is an asymptotically stable constant
steady state if $d_{1}<d_{2}$ and%
\begin{equation}
\left \{ 
\begin{array}{lll}
\lambda _{1}d_{1}\geq F_{0}, & \text{or} & \medskip \\ 
\lambda _{1}d_{1}<F_{0} & \text{and} & 0<d_{2}<\tilde{d},%
\end{array}%
\right.  \label{3.7}
\end{equation}%
where%
\begin{equation}
d_{i}=\sigma b\frac{\alpha }{1+\alpha ^{2}}\frac{\left( \lambda
_{i}d_{1}+5\right) }{\left( F_{0}-\lambda _{i}d_{1}\right) \lambda _{i}},
\label{3.8}
\end{equation}%
and%
\begin{equation}
\tilde{d}=\min_{i\geq 0}d_{i}.  \label{3.9}
\end{equation}%
If\ $d_{1}>d_{2}$, the euilibrium $\left( u^{\ast },v^{\ast }\right) $\ is
asymptotically stable if $\lambda _{1}d_{1}\geq F_{0}$\ and the eigenvalues%
\begin{equation}
\xi _{1,2}\left( \lambda _{i}\right) =\frac{1}{2}\left[ \text{tr}J_{i}\pm i%
\sqrt{4\det J_{i}-\left( \text{tr}J_{i}\right) ^{2}}\right]  \label{3.11}
\end{equation}%
satisfy%
\begin{equation}
\left \vert \arg \left( \xi _{1}\left( \lambda _{i}\right) \right) \right
\vert >\delta \frac{\pi }{2}\text{ and }\left \vert \arg \left( \xi
_{2}\left( \lambda _{i}\right) \right) \right \vert >\delta \frac{\pi }{2}
\label{3.12}
\end{equation}%
for all $\lambda _{i}\in \left( \lambda _{01},\lambda _{02}\right) $.
\end{proposition}

\begin{proof}
In order to study the local asymptotic stability in the PDE sense, we will
linearize the system. Following the standard linear operator theory (see 
\cite{Casten1977}), and keeping in mind the fractional nature of the system,
we can state that $\left( u^{\ast },v^{\ast }\right) $ is asymptotically
stable if the eigenvalues of the linearized system satisfy the conditions of
Lemma \ref{Lemma3}.

Suppose that $\left( \phi \left( x\right) ,\psi \left( x\right) \right) $ is
an eigenfunction of $L$ corresponding to the eigenvalue $\xi $. Then,%
\begin{equation*}
\left( 
\begin{array}{cc}
d_{1}\Delta +F_{0}-\xi \left( \lambda _{i}\right) & F_{1} \\ 
\sigma G_{0} & d_{2}\Delta +\sigma G_{1}-\xi \left( \lambda _{i}\right)%
\end{array}%
\right) \left( 
\begin{array}{c}
\phi \\ 
\psi%
\end{array}%
\right) =\left( 
\begin{array}{c}
0 \\ 
0%
\end{array}%
\right) .
\end{equation*}%
With%
\begin{equation*}
\phi =\sum_{0\leq i\leq \infty ,1\leq j\leq m_{i}}a_{ij}\Phi _{ij}\text{ and 
}\psi =\sum_{0\leq i\leq \infty ,1\leq j\leq m_{i}}b_{ij}\Phi _{ij},
\end{equation*}%
we obtain%
\begin{equation*}
\sum_{0\leq i\leq \infty ,1\leq j\leq m_{i}}\left( 
\begin{array}{cc}
F_{0}-d_{1}\lambda _{i}-\xi \left( \lambda _{i}\right) & F_{1} \\ 
\sigma G_{0} & \sigma G_{1}-d_{2}\lambda _{i}-\xi \left( \lambda _{i}\right)%
\end{array}%
\right) \left( 
\begin{array}{c}
a_{ij} \\ 
b_{ij}%
\end{array}%
\right) \Phi _{ij}=\left( 
\begin{array}{c}
0 \\ 
0%
\end{array}%
\right) .
\end{equation*}%
It holds that%
\begin{equation*}
\left( 
\begin{array}{cc}
F_{0}-d_{1}\lambda _{i}-\xi \left( \lambda _{i}\right) & F_{1} \\ 
\sigma G_{0} & \sigma G_{1}-d_{2}\lambda _{i}-\xi \left( \lambda _{i}\right)%
\end{array}%
\right) =J_{i}-\xi \left( \lambda _{i}\right) I,
\end{equation*}%
with $J_{i}$ as defined in (\ref{3.6.1}). The characteristic equation of
matrix $J_{i}$ is%
\begin{equation}
\xi ^{2}\left( \lambda _{i}\right) -\text{tr}J_{i}\text{ }\xi \left( \lambda
_{i}\right) +\det J_{i}=0,  \label{3.15}
\end{equation}%
where%
\begin{equation*}
\text{tr}J_{i}=-\left( d_{1}+d_{2}\right) \lambda _{i}+\text{tr}J,
\end{equation*}%
and%
\begin{equation*}
\det J_{i}=\left( \lambda _{i}d_{1}-F_{0}\right) \lambda _{i}d_{2}+\frac{%
\sigma b\alpha }{1+\alpha ^{2}}\left( \lambda _{i}d_{1}+5\right) .
\end{equation*}%
In order to investigate the stability of $\left( u^{\ast },v^{\ast }\right) $%
, we examine the nature of the eigenvalues by taking the discriminant of (%
\ref{3.15}), which is given by%
\begin{align*}
\Upsilon _{i}& =\left( \text{tr}J_{i}\right) ^{2}-4\det J_{i}\medskip \\
& =\left( d_{1}-d_{2}\right) ^{2}\lambda _{i}^{2}+2\left( d_{1}-d_{2}\right)
\left( -F_{0}+\sigma G_{1}\right) \lambda _{i}+\left( \left( -F_{0}+\sigma
G_{1}\right) ^{2}+4\sigma F_{1}G_{0}\right) \medskip \\
& =\left( d_{1}-d_{2}\right) ^{2}\lambda _{i}^{2}+2\left( d_{1}-d_{2}\right)
\left( -F_{0}+\sigma G_{1}\right) \lambda _{i}+\Upsilon .
\end{align*}%
The sign of $\Upsilon _{i}$ is important for the stability of $\left(
u^{\ast },v^{\ast }\right) $. The discriminant of $\Upsilon _{i}$ with
respect to $\lambda _{i}$ is%
\begin{equation*}
\Delta _{\lambda }=32\left( d_{1}-d_{2}\right) ^{2}\sigma b\frac{\alpha ^{3}%
}{\left( 1+\alpha ^{2}\right) ^{2}}.
\end{equation*}%
We have a number of cases for $\Delta _{\lambda }$:

\begin{itemize}
\item If $d_{1}=d_{2}$, we notice that 
\begin{equation*}
\Upsilon _{i}=\Upsilon _{0}=\Upsilon .
\end{equation*}%
Hence, the exact same conditions for OFDE stability as described in
Proposition \ref{Prop1} apply here.

\item If $d_{1}\not =d_{2}$, then $\Delta_{\lambda}>0$. Hence, $\Upsilon_{i}$
has two real roots and we have two cases:

\begin{itemize}
\item If $d_{1}<d_{2}$, then using tr$J_{i}>0$, we have%
\begin{equation*}
2\left( d_{1}-d_{2}\right) \left( -F_{0}+\sigma G_{1}\right) >0.
\end{equation*}%
Thus, since $\Upsilon >0$, the solutions $\lambda _{01}$ and $\lambda _{02}$
of the equation $\Upsilon _{i}=0$ are both negative regardless of $i$.
Hence, $\Upsilon _{i}>0$ for all $i$ and the roots of (\ref{3.15})%
\begin{equation*}
\xi _{1}\left( \lambda _{i}\right) =\frac{\text{tr}J_{i}-\sqrt{\left( \text{%
tr}J_{i}\right) ^{2}-4\det J_{i}}}{2},
\end{equation*}%
and%
\begin{equation*}
\xi _{2}\left( \lambda _{i}\right) =\frac{\text{tr}J_{i}+\sqrt{\left( \text{%
tr}J_{i}\right) ^{2}-4\det J_{i}}}{2}.
\end{equation*}%
are real. Note that 
\begin{equation*}
\text{tr}J<0\Rightarrow \text{tr}J_{i}<0,
\end{equation*}%
which leads to $\xi _{1}\left( \lambda _{i}\right) <0$. Also, if $\lambda
_{1}d_{1}\geq F_{0}$, then $\xi _{2}\left( \lambda _{i}\right) <0$. This
leads to 
\begin{equation*}
\left \vert \arg \left( \xi _{1}\left( \lambda _{i}\right) \right) \right
\vert =\left \vert \arg \left( \xi _{2}\left( \lambda _{i}\right) \right)
\right \vert =\pi ,
\end{equation*}%
which guarantees the asymptotic stability of $\left( u^{\ast },v^{\ast
}\right) $.

Alternatively, if $\lambda _{1}d_{1}<F_{0}$ and $0<d_{2}<\tilde{d}$, then%
\begin{equation*}
\lambda _{i}d_{1}<F_{0}\text{ and }d_{2}<d_{i}\text{ for }i\in \left[
1,i_{\alpha }\right] .
\end{equation*}%
It follows that $\det J_{i}>0$ for all $i\in \left[ 1,i_{\alpha }\right] $.
Furthermore, if $i>i_{\alpha }$ then $\lambda _{i}d_{1}\geq F_{0}$ and $\det
J_{i}>0$. The argument leads to the asymptotic stability of $(u^{\ast
},v^{\ast })$ again.

\item If $d_{1}>d_{2}$, we have%
\begin{equation*}
2\left( d_{1}-d_{2}\right) \left( -F_{0}+\sigma G_{1}\right) >0,
\end{equation*}%
and since $\Upsilon >0$, we have $0<\lambda _{01}\leq \lambda _{02}$. Hence,%
\begin{equation*}
\left \{ 
\begin{array}{c}
\lambda _{i}\geq \lambda _{02}\medskip \\ 
\text{or}\medskip \\ 
\lambda _{i}\leq \lambda _{01}%
\end{array}%
\right. \Rightarrow \Upsilon _{i}\geq 0,
\end{equation*}%
which takes us back to the previous case. Again, for $\lambda _{1}d_{1}\geq
F_{0}$, we have $\det J_{i}>0$ and thus $\xi _{1}$ and $\xi _{2}$ are
negative. Next, if $\lambda _{01}<\lambda _{i}<\lambda _{02}$, we have $%
\Upsilon _{i}<0$ and $\det J_{i}>0$. The eigenvalues are, thus, complex, see
(\ref{3.11}). Hence, $\left( u^{\ast },v^{\ast }\right) $\ is an
asymptotically stable equilibrium subject to (\ref{3.12}) for all $\lambda
_{i}$ in the interval $\left( \lambda _{01},\lambda _{02}\right) $.
\end{itemize}
\end{itemize}
\end{proof}

\subsection{Global Stability}

In this section, we derive conditions for the global asymptotic stability.
First of all, let us define the function%
\begin{equation}
f_{a}\left( u\right) =\frac{a-u}{\varphi \left( u\right) },  \label{4.1}
\end{equation}%
where%
\begin{equation}
\varphi \left( u\right) =\frac{u}{1+u^{2}}.  \label{4.2}
\end{equation}%
Obviously, we have%
\begin{equation}
f_{a}\left( u^{\ast }\right) =\frac{4\alpha }{\varphi \left( \alpha \right) }%
.  \label{4.3}
\end{equation}%
Also, setting%
\begin{equation}
U=u-u^{\ast }\text{ and }V=v-v^{\ast },  \label{4.4}
\end{equation}%
we obtain the modified system%
\begin{equation}
\left\{ 
\begin{array}{l}
_{t_{0}}^{C}D_{t}^{\delta }U-d_{1}\Delta U=\varphi \left( U+u^{\ast }\right) 
\left[ \left( f_{a}\left( U+u^{\ast }\right) -f_{a}\left( u^{\ast }\right)
\right) -4V\right] ,\medskip  \\ 
_{t_{0}}^{C}D_{t}^{\delta }V-d_{2}\Delta V=\sigma b\varphi \left( U+u^{\ast
}\right) \left[ U\left( U+2u^{\ast }\right) -V\right] .%
\end{array}%
\right.   \label{4.5}
\end{equation}

\begin{theorem}
Subject to%
\begin{equation}
0<a^{2}\leq27,  \label{4.6}
\end{equation}
equilibrium $\left( u^{\ast},v^{\ast}\right) $ is globally asymptotically
stable.
\end{theorem}

\begin{proof}
In order to establish the global asymptotic stability, we use the Lyapunov
method. Let%
\begin{equation}
L\left( t\right) =\int_{\Omega }\left[ \frac{\sigma b}{3}U^{3}+\sigma
bu^{\ast }U^{2}+2V^{2}\right] dx.  \label{4.7}
\end{equation}%
Taking the fractional Caputo derivative of (\ref{4.7}) and using (\ref{1.8}%
), we obtain%
\begin{align*}
_{t_{0}}^{C}D_{t}^{\delta }L\left( t\right) & =\int_{\Omega }\left[ \left( 
\frac{\sigma b}{3}\right) \text{ }_{t_{0}}^{C}D_{t}^{\alpha }U^{3}+\left(
\sigma bu^{\ast }\right) \ _{t_{0}}^{C}D_{t}^{\alpha
}U^{2}+2_{t_{0}}^{C}D_{t}^{\alpha }V^{2}\right] dx\medskip \\
& \leq \int_{\Omega }\left[ \sigma bU^{2}{}_{t_{0}}^{C}D_{t}^{\alpha
}U+2\left( \sigma bu^{\ast }\right) U\ _{t_{0}}^{C}D_{t}^{\alpha }U+4V\
_{t_{0}}^{C}D_{t}^{\alpha }V\right] dx,
\end{align*}%
see \cite{Alsaedi2015}. Further simplification yields%
\begin{align*}
_{t_{0}}^{C}D_{t}^{\delta }L\left( t\right) & \leq \int_{\Omega }\left[
\sigma bU\left( U+2u^{\ast }\right) \ _{t_{0}}^{C}D_{t}^{\alpha }U+4V\
_{t_{0}}^{C}D_{t}^{\alpha }V\right] dx\medskip \\
& \leq \int_{\Omega }\varphi \left( U+u^{\ast }\right) \left \{ \sigma
bU\left( U+2u^{\ast }\right) \left[ \left( f_{a}\left( U+u^{\ast }\right)
-f_{a}\left( u^{\ast }\right) \right) -4V\right] \right. \\
& \ \ \ \ \ \left. +4V\sigma b\left[ U\left( U+2u^{\ast }\right) -V\right]
\right \} dx+\int_{\Omega }\sigma bU\left( U+2u^{\ast }\right) d_{1}\Delta
Udx \\
& \ \ \ \ \ +\int_{\Omega }4Vd_{2}\Delta Vdx\medskip \\
& \leq \int_{\Omega }\sigma b\varphi \left( U+u^{\ast }\right) \left \{
U\left( U+2u^{\ast }\right) \left( f_{a}\left( U+u^{\ast }\right)
-f_{a}\left( u^{\ast }\right) \right) \right. \\
& \ \ \ \ \ \left. -4U\left( U+2u^{\ast }\right) V+4VU\left( U+2u^{\ast
}\right) -4V^{2}\right \} dx \\
& \ \ \ \ \ +\sigma b\int_{\Omega }U\left( U+2u^{\ast }\right) d_{1}\Delta
Udx+4d_{2}\int_{\Omega }V\Delta Vdx,
\end{align*}%
leading to%
\begin{align}
_{t_{0}}^{C}D_{t}^{\delta }L\left( U,V\right) & \leq \underset{I_{1}\left(
t\right) }{\underbrace{\sigma b\int_{\Omega }\varphi \left( U+u^{\ast
}\right) \left \{ U\left( U+2u^{\ast }\right) \left( f_{a}\left( U+u^{\ast
}\right) -f_{a}\left( u^{\ast }\right) \right) -4V^{2}\right \} dx}}+  \notag
\\
& \ \ \ \ \ +\underset{I_{2}\left( t\right) }{\underbrace{\sigma
bd_{1}\int_{\Omega }U\left( U+2u^{\ast }\right) \Delta
Udx+4d_{2}\int_{\Omega }V\Delta Vdx}}.  \label{4.8}
\end{align}

We note that the function $f_{a}$ is strictly decreasing over the interval $%
\left( 0,a\right) $ when $0<a^{2}\leq 27$. Hence, by the mean value theorem,
there exists some $c$ between $u$ and $u^{\ast }$ such that%
\begin{equation*}
f_{a}\left( U+u^{\ast }\right) -f_{a}\left( u^{\ast }\right) =Uf_{a}^{\prime
}\left( c\right) .
\end{equation*}%
Substituting in $I_{1}\left( t\right) $ yields%
\begin{equation*}
I_{1}\left( t\right) =\sigma b\int_{\Omega }\varphi \left( U+u^{\ast
}\right) \left \{ U^{2}\left( U+2u^{\ast }\right) f_{a}^{\prime }\left(
c\right) -4V^{2}\right \} <0.
\end{equation*}%
For $I_{2}\left( t\right) $, we have%
\begin{align*}
I_{2}\left( t\right) & =\sigma bd_{1}\int_{\Omega }U\left( U+2u^{\ast
}\right) \Delta Udx+4d_{2}\int_{\Omega }V\Delta Vdx\medskip \\
& =-\sigma bd_{1}\int_{\Omega }\nabla \left( U^{2}+2u^{\ast }U\right) \nabla
Udx-4d_{2}\int_{\Omega }\left \vert \nabla V\right \vert ^{2}dx\medskip \\
& =-\sigma bd_{1}\int_{\Omega }2\left( U+u^{\ast }\right) \left \vert \nabla
U\right \vert ^{2}dx-4d_{2}\int_{\Omega }\left \vert \nabla V\right \vert
^{2}dx<0.
\end{align*}%
Hence,%
\begin{equation*}
_{t_{0}}^{C}D_{t}^{\delta }L\left( U,V\right) <0
\end{equation*}%
and $_{t_{0}}^{C}D_{t}^{\delta }L\left( t\right) =0$ if and only if $\left(
U,V\right) =\left( 0,0\right) $. Therefore, by the direct Lyapunov method,
the constant steady state $\left( u^{\ast },v^{\ast }\right) $ is globally
asymptotically stable subject to (\ref{4.6}).
\end{proof}

\section{Numerical Examples\label{SecNum}}

In this section, we present some numerical examples to show the effect of $%
\delta $ on the dynamics of the fractional Lengyel--Epstein system (\ref{2.1}%
). Consider the parameter set $\left( a,b,\sigma ,d_{1},d_{2}\right) =\left(
15,1,7,1,10\right) $ and initial conditions%
\begin{equation}
\left \{ 
\begin{array}{l}
u\left( x,0\right) =1+0.3\sin \left( \frac{x}{2}\right) ,\medskip \\ 
v\left( x,0\right) =2+0.6\sin \left( \frac{x}{2}\right) .%
\end{array}%
\right.  \label{5.1}
\end{equation}%
The solutions of system (\ref{2.1}) with zero Neumann boundary conditions
and different values of $\delta $ were obtained numerically for $t\in \left[
0,10\right] $ and $x\in \left[ 0,20\right] $ with $\Delta t=0.001$ and $%
\Delta x=0.5$. Figures \ref{Fig1} and \ref{Fig2} show the one--dimensional
spatio--temporal states $u\left( x,t\right) $ and $v\left( x,t\right) $,
respectively. We see that for $\delta =1$, the solution is oscillatory in
nature and thus asymptotically unstable. This is confirmed by means of the
phase--space plot taken at a single spatial point $x=10$ as depicted in
Figure \ref{Fig3}. The solution converges to an ellipse signifying a
periodic nature. As $\delta $ is made smaller, the solution becomes
asymptotically stable and converges to the unique spatially homogeneous
constant steady state 
\begin{equation}
\left( u^{\ast },v^{\ast }\right) =\left( \frac{a}{5},1+\left( \frac{a}{5}%
\right) ^{2}\right) =\left( 3,10\right) .  \label{5.2}
\end{equation}%
Furthermore, we see that the smaller $\delta $, the faster the solution
converges to the steady state. This strong dependence of the asymptotic
stability on $\delta $ is very interesting as it gives us a new perspective
into the control and dynamics of the CIMA\ chemical reaction.

\begin{figure}[tbh]
\centering \includegraphics[width=5in]{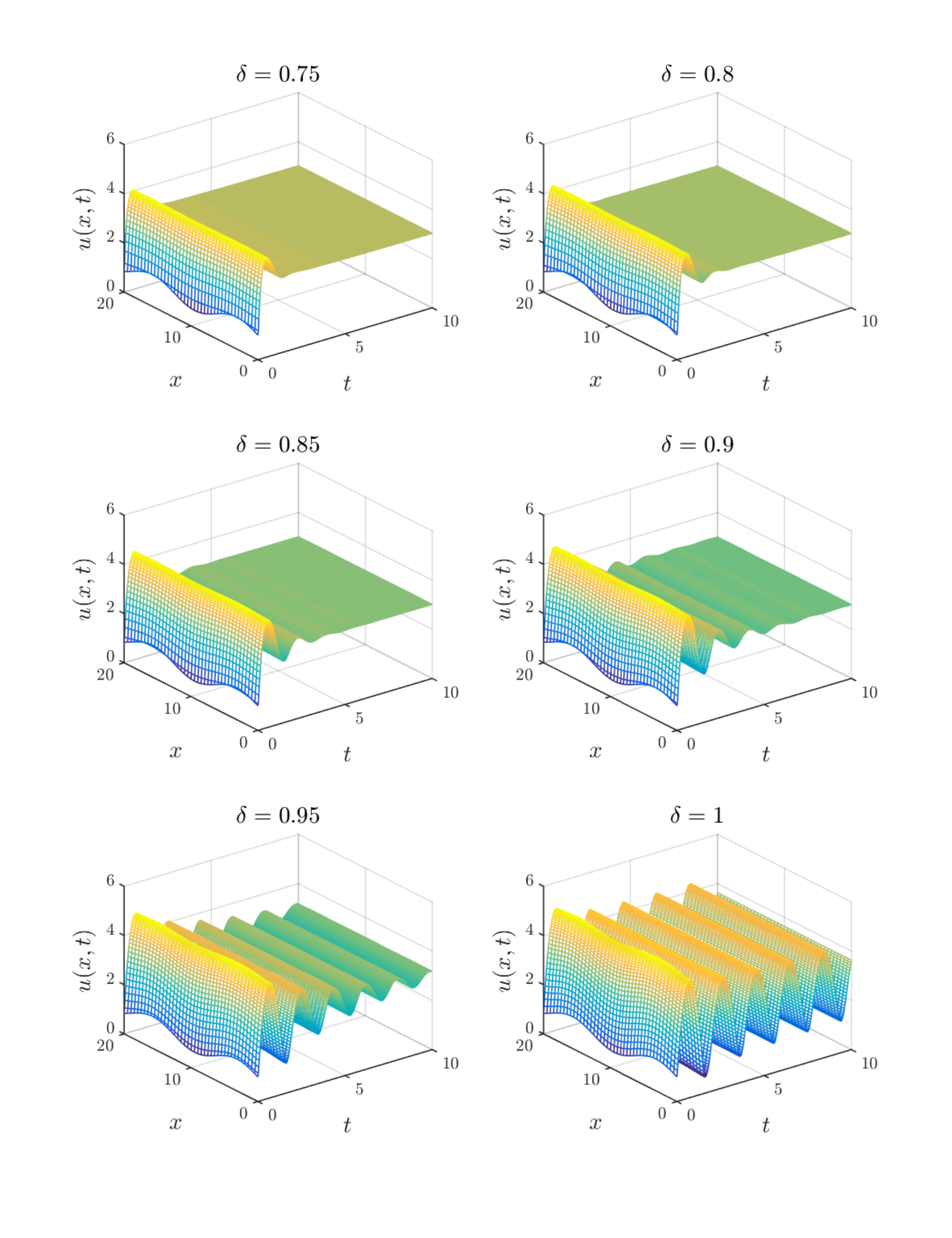}
\caption{One dimensional concentration $u\left( x,t\right) $ as a solution
of (\protect\ref{2.1}) with $\left( a,b,\protect\sigma,d_{1},d_{2}\right)
=\left( 15,1,7,1,10\right) \,$, initial conditions (\protect\ref{5.1}), zero
Nuemann boundaries, and different values for $\protect\delta$.}
\label{Fig1}
\end{figure}

\begin{figure}[tbh]
\centering \includegraphics[width=5in]{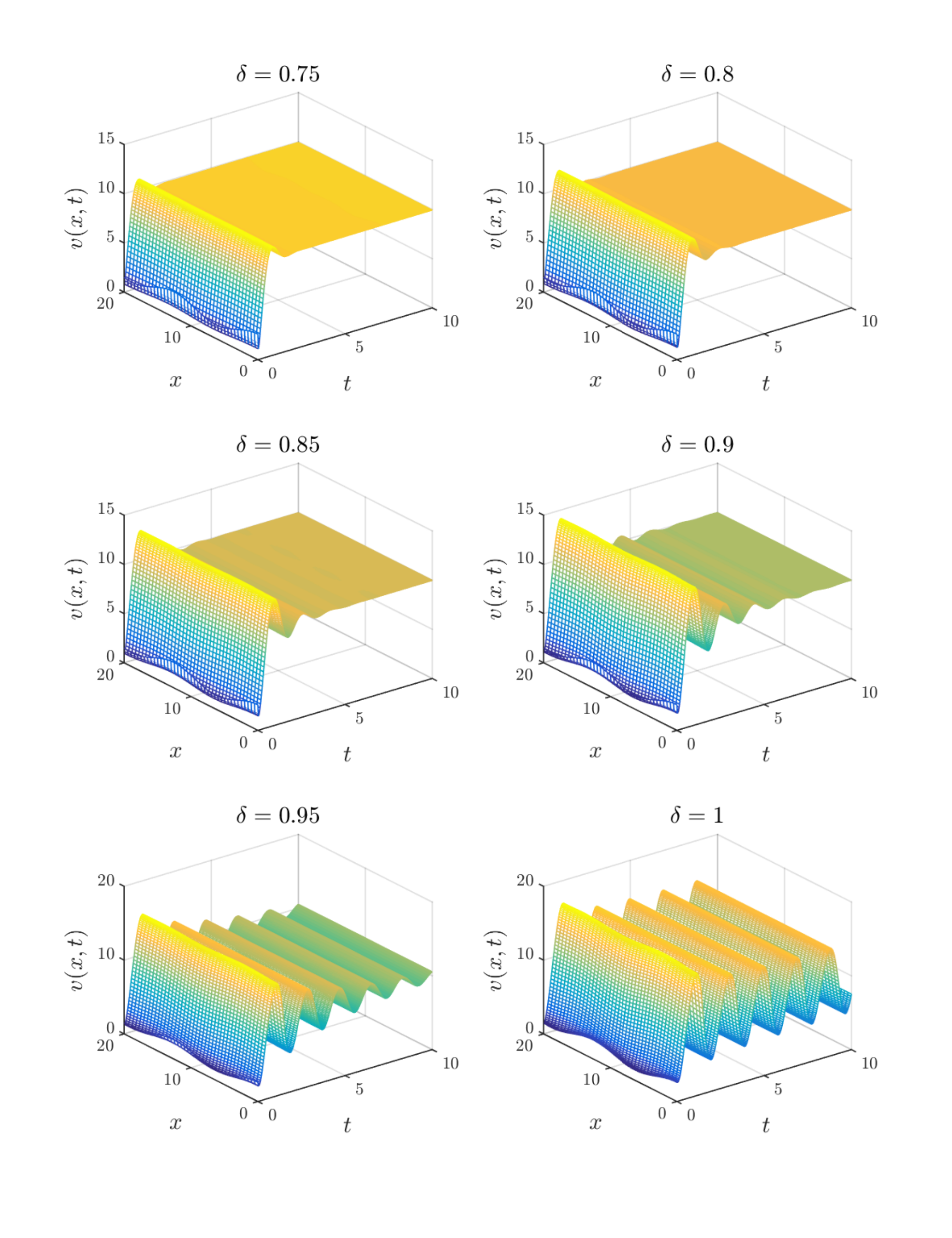}
\caption{One dimensional concentration $v\left( x,t\right) $ as a solution
of (\protect\ref{2.1}) with $\left( a,b,\protect\sigma,d_{1},d_{2}\right)
=\left( 15,1,7,1,10\right) \,$, initial conditions (\protect\ref{5.1}), zero
Nuemann boundaries, and different values for $\protect\delta$.}
\label{Fig2}
\end{figure}

\begin{figure}[tbh]
\centering \includegraphics[width=5in]{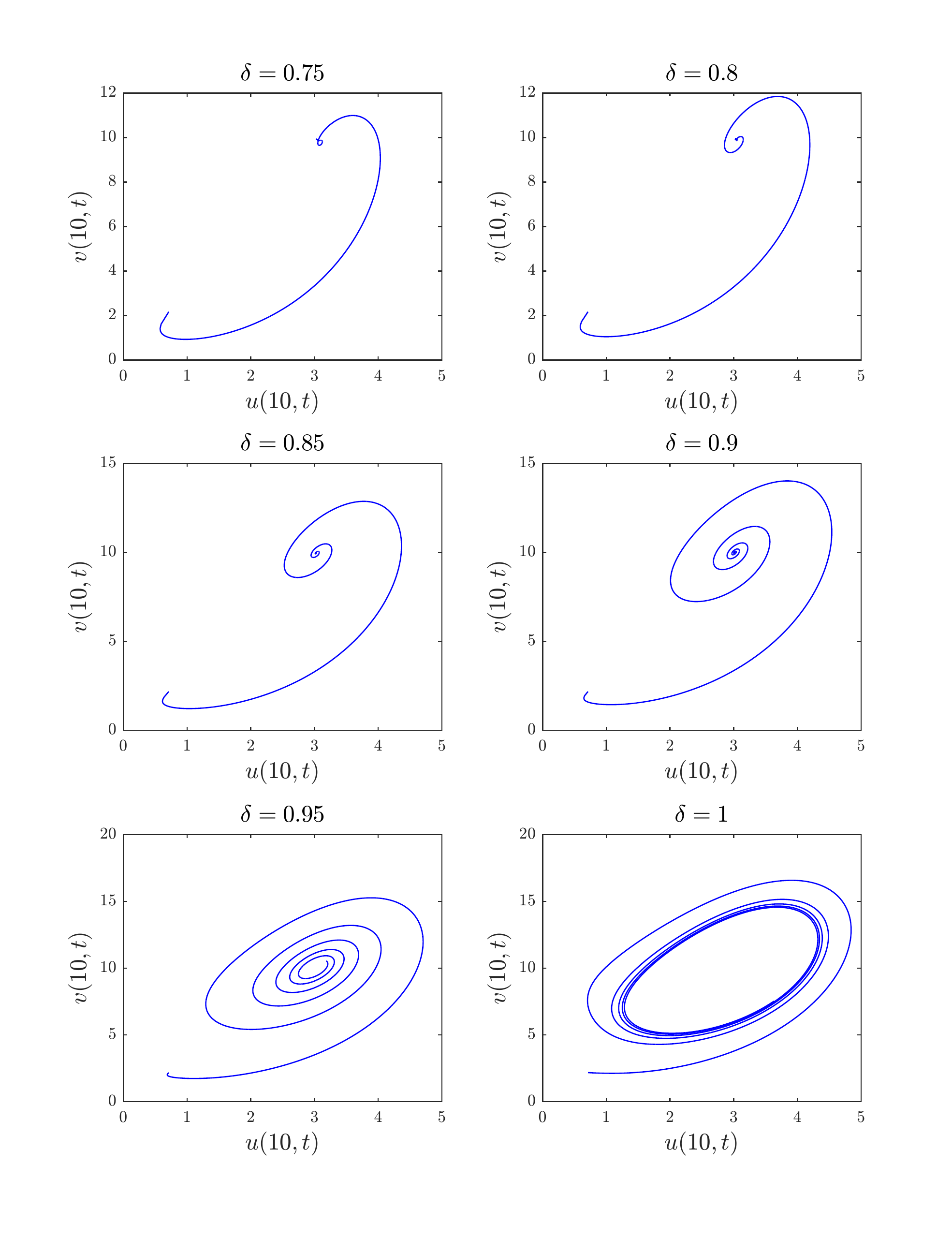}
\caption{Phase plot of system (\protect\ref{2.1}) taken at $x=10$ with $%
\left( a,b,\protect\sigma ,d_{1},d_{2}\right) =\left( 15,1,7,1,10\right) \,$%
, initial conditions (\protect\ref{5.1}), zero Nuemann boundaries, and
different values for $\protect\delta $.}
\label{Fig3}
\end{figure}

In addition to these one--dimensional examples, we have also examined the
two--dimensional case. We consider the parameter set $\left( a,b,\sigma
,d_{1},d_{2}\right) =\left( 15,1.2,8,1,24\right) $ with initial conditions%
\begin{equation}
\left \{ 
\begin{array}{l}
u\left( x,y,0\right) =3.5\left( 1+0.2w_{u}\left( x,y\right) \right) ,\medskip
\\ 
v\left( x,y,0\right) =10.5\left( 1+0.2w_{v}\left( x,y\right) \right) .%
\end{array}%
\right.  \label{5.3}
\end{equation}%
with $w_{u}\left( x,y\right) $ and $w_{v}\left( x,y\right) $ being Gaussian
distributed random functions with zero mean and unit variance. Figure \ref%
{Fig4} shows snap shots of the concentrations $u\left( x,y,t\right) $ and $%
v\left( x,y,t\right) $ taken at time instances $t=0$, $t=5$, and $t=20$ with 
$\delta =1$. We see that the diffusion--driven or Turing instability leads
to the formation of patterns in the form of dots and stripes. Reducing the
fractional order to $\delta =0.98$ leads to a different type of patterns as
shown in Figure \ref{Fig5}. This means that the fractional order has an
impact on the Turing patterns evolving over time, which is an interesting
observation. Reducing the fractional order further to $\delta =0.95$ also
yields slightly different patterns as shown in Figure \ref{Fig6}.

\begin{figure}[tbh]
\centering \includegraphics[width=4in]{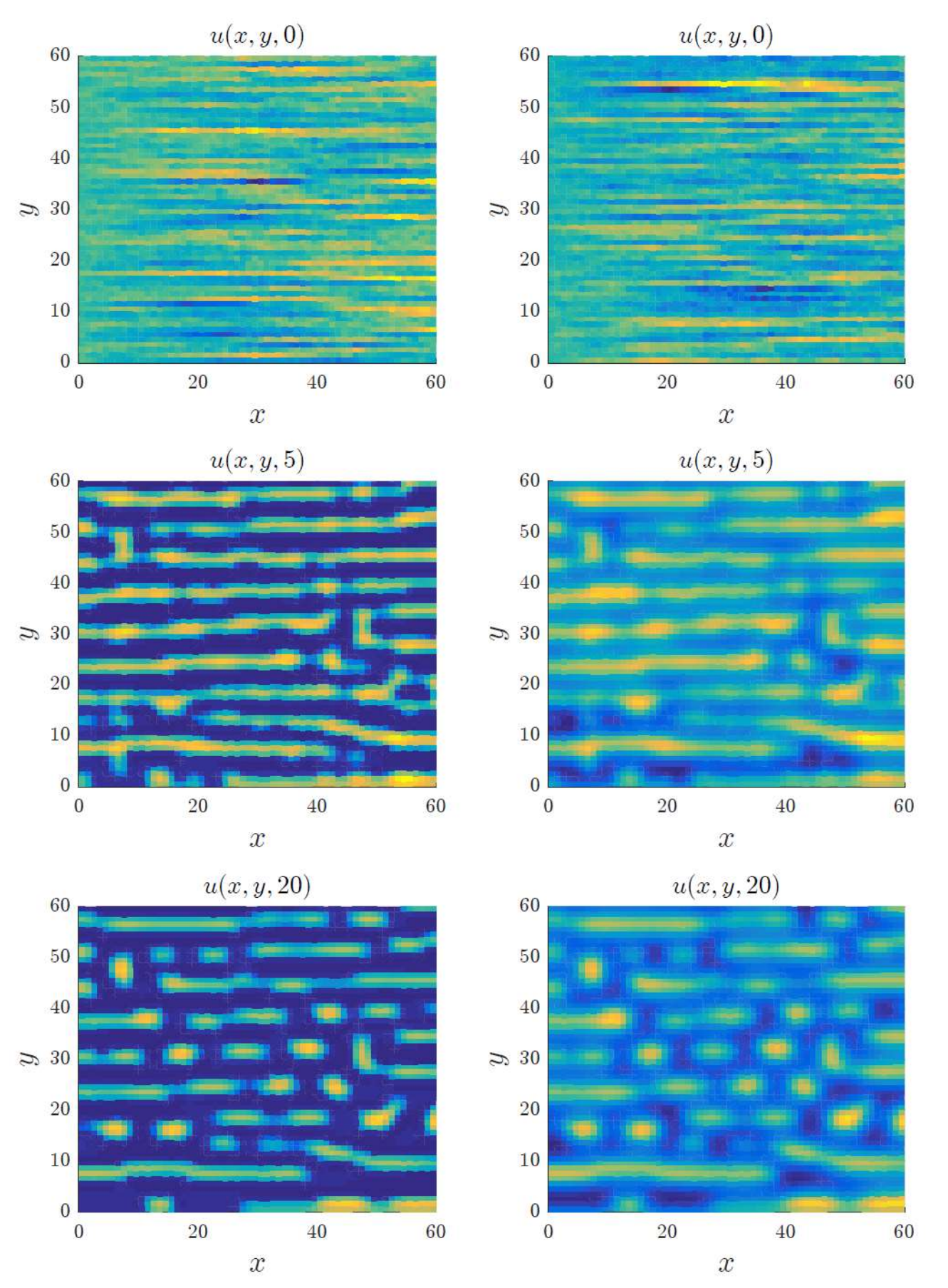}
\caption{Two dimensional concentrations $u\left( x,y,t\right) $ and $v\left(
x,y,t\right) $ for $\left( a,b,\protect\sigma ,d_{1},d_{2}\right) =\left(
15,1.2,8,1,24\right) \,$, initial conditions (\protect\ref{5.3}), zero
Nuemann boundaries, and $\protect\delta =1$.}
\label{Fig4}
\end{figure}

\begin{figure}[tbh]
\centering \includegraphics[width=4in]{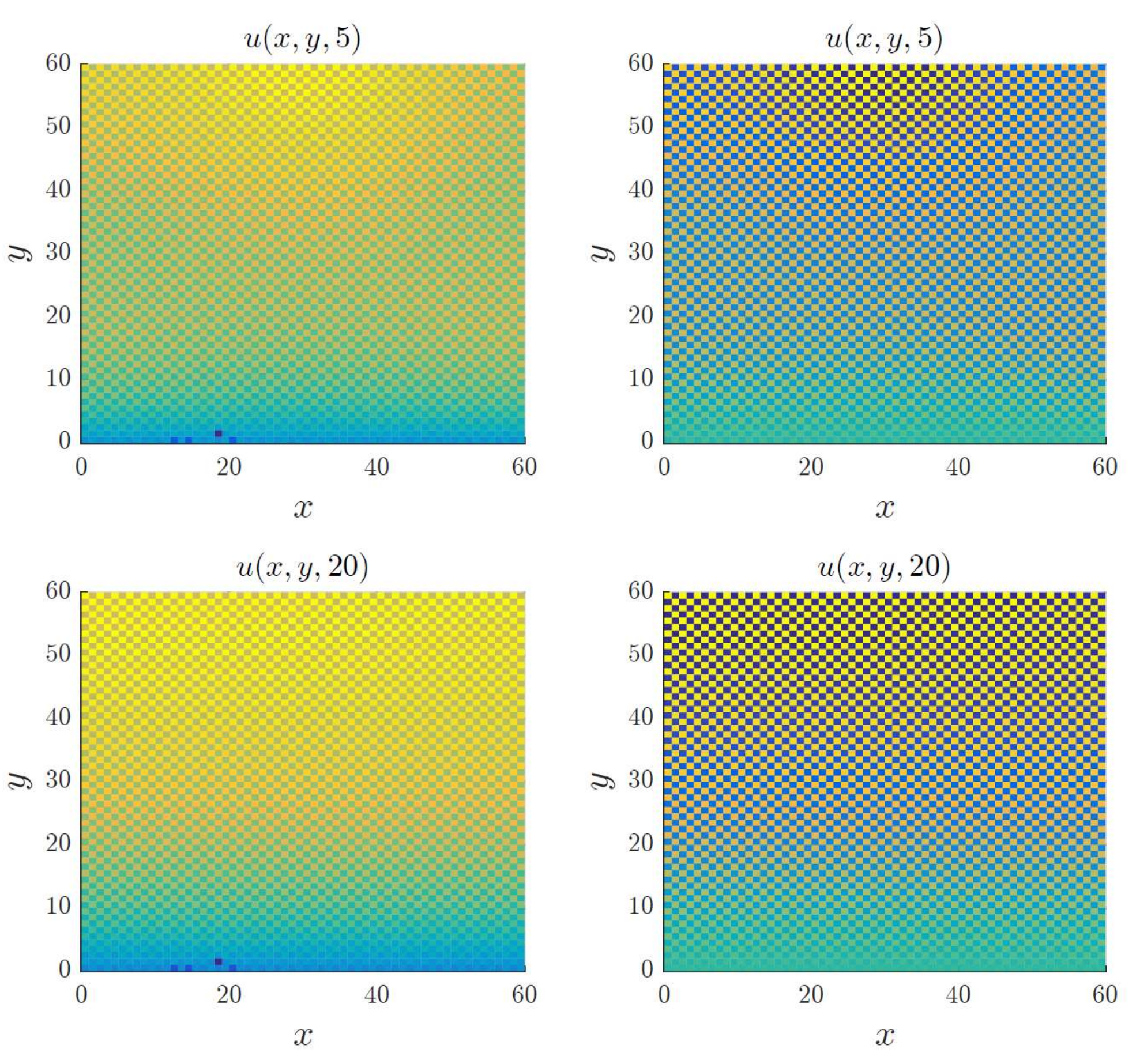}
\caption{Two dimensional concentrations $u\left( x,y,t\right) $ and $v\left(
x,y,t\right) $ for $\left( a,b,\protect\sigma ,d_{1},d_{2}\right) =\left(
15,1.2,8,1,24\right) \,$, initial conditions (\protect\ref{5.3}), zero
Nuemann boundaries, and $\protect\delta =0.98$.}
\label{Fig5}
\end{figure}

\begin{figure}[tbh]
\centering \includegraphics[width=4in]{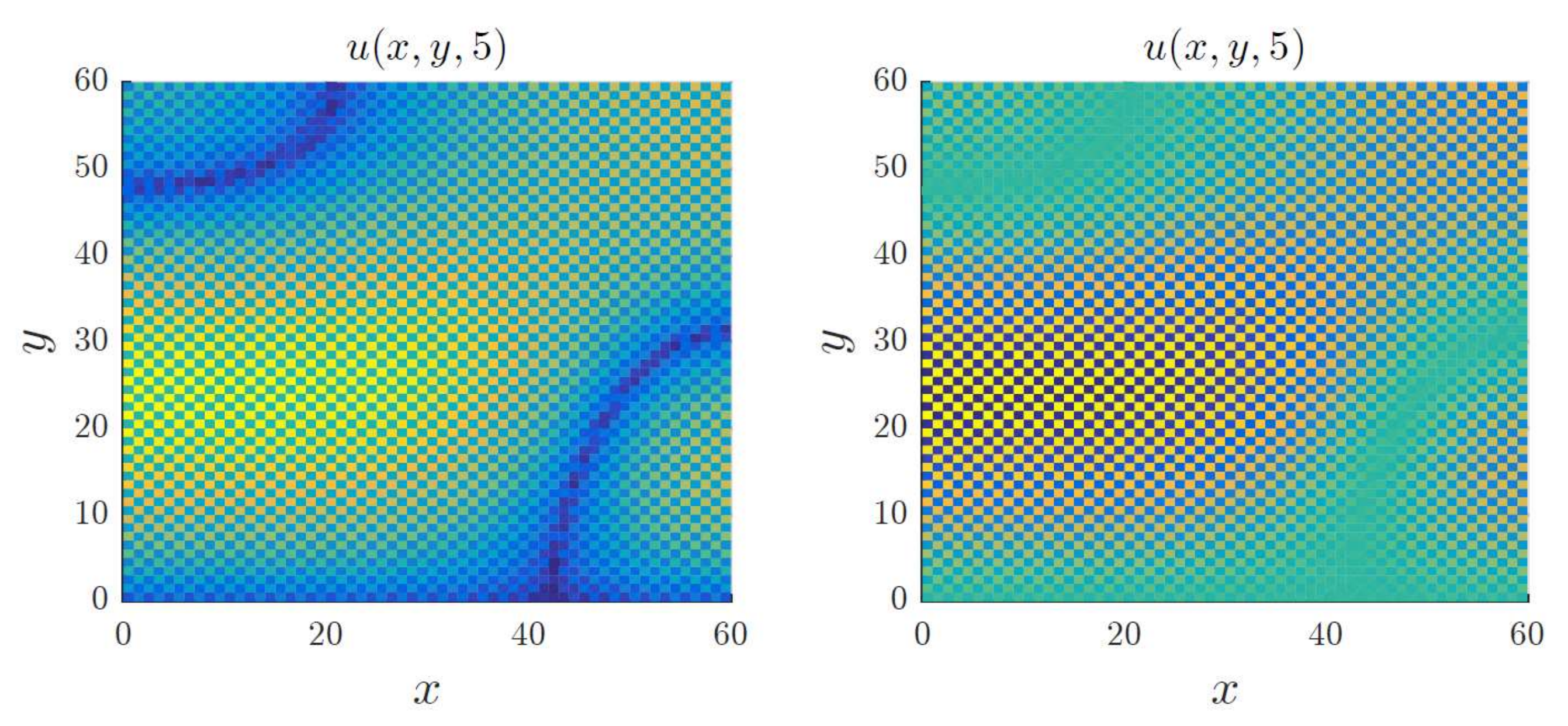}
\caption{Two dimensional concentrations $u\left( x,y,t\right) $ and $v\left(
x,y,t\right) $ for $\left( a,b,\protect\sigma ,d_{1},d_{2}\right) =\left(
15,1.2,8,1,24\right) \,$, initial conditions (\protect\ref{5.3}), zero
Nuemann boundaries, and $\protect\delta =0.95$.}
\label{Fig6}
\end{figure}

\section{Concluding Remarks\label{SecSum}}

In this paper, we have considered a time--fractional version of the
Lengyel--Epstein system modeling the chlorite--iodide malonic acid (CIMA)
chemical reaction. The Lengyel--Epstein model is well known for exhibiting
Turing patterns, which makes it of interest to researchers in mathematics,
chemistry, and biology. Introducing fractional time derivatives has recently
been shown to model natural phenomena more accurately especially in chemical
reactions. We have established sufficient conditions for the local
asymptotic stability of the system's unique equilibrium in the ODE and PDE
senses through the linearization method. In addition, we have employed the
direct Lyapunov method to establish the global asymptotic stability of the
steady state solution.

Through numerical investigation, we have seen that a periodic solution in
the standard case, which corresponds to pattern formation, became
asymptotically stable when the differentiation order decreased below $1$.
This is an important observation that requires closer investigation and
analysis as it provides a new perspective into the control and applications
of the Lengyel--Epstein system. We have also seen that the presence of
diffusion alters the stability conditions of the system, which is not at all
unlike the standard case. Furthermore, we saw that the type of patterns that
form as a result of the diffusion--driven instability changes as the
fractional order is varied. More investigation will be performed in future
studies to explore these observations.

\section*{Acknowledgment}

The authors would like to thank Prof. M. Kirane of La Rochelle University in
France for his continued assistance and guidance that led to the conclusion
of this research.

\end{document}